\documentclass[a4paper]{amsart}
\usepackage{amssymb}
\usepackage{amsmath}
\usepackage{xcolor}

\usepackage{tikz-cd}
\tikzcdset{arrow style=math font}

\usepackage{hyperref}
\hypersetup{colorlinks=true,allcolors=black}

\newtheorem{thm}{Theorem}[section]
\newtheorem*{thm*}{Theorem}

\newtheorem{cor}[thm]{Corollary}
\newtheorem*{cor*}{Corollary}

\newtheorem{prop}[thm]{Proposition}
\newtheorem{thmx}{Theorem}

\theoremstyle{definition}
\newtheorem{defn}[thm]{Definition}
\newtheorem{ex}[thm]{Example}

\theoremstyle{remark}

\newcommand{\ZZ}{{\mathbb  Z}}
\newcommand{\RR}{{\mathbb  R}}

\newcommand{\whM}{{ \hat M}}

\newcommand{\mft}{\mathfrak{t}}

\newcommand{\id}{\operatorname{id}}

\makeatletter
\newsavebox{\@brx}
\newcommand{\llangle}[1][]{\savebox{\@brx}{\(\m@th{#1\langle}\)}%
  \mathopen{\copy\@brx\kern-0.5\wd\@brx\usebox{\@brx}}}
\newcommand{\rrangle}[1][]{\savebox{\@brx}{\(\m@th{#1\rangle}\)}%
  \mathclose{\copy\@brx\kern-0.5\wd\@brx\usebox{\@brx}}}
\makeatother

\numberwithin{equation}{section}

\newcommand{\cm}[1]{\marginpar{*}{\tiny #1 }}

\newcommand\actionarrow[2]{#1 \begin{tikzcd}[ampersand replacement=\&, cramped, sep=small] \& \arrow[loop left] \phantom{X} \end{tikzcd} \hspace{-9pt} #2}

\begin{document}

\title{Polar and Torus Manifolds}

\author{Francisco C. Caramello Jr.}
\address{Departamento de Matemática, Universidade Federal de Santa Catarina, R. Eng. Agr. Andrei Cristian Ferreira, 88040-900, Florianópolis - SC, Brazil}
\email{francisco.caramello@ufsc.br}

\author{Dirk T\"oben}
\address{Dirk T\"oben, Departamento de Matemática, Universidade Federal de São Carlos, Brazil}
\email{dirktoben@ufscar.br}

\date{}

\subjclass[2020]{Primary 57S12, 53C20}

\begin{abstract}
We show that quasitoric manifolds can be equipped with an invariant metric for which the torus action is polar. More generally, we show that an infinitesimally polar action of a torus admits an invariant Riemannian metric making it polar if and only if the real orbifold Euler class of its quotient vanishes. Because the orbit spaces of quasitoric manifolds are contractible simple polytopes, this cohomological obstruction trivially vanishes.
\end{abstract}

\maketitle

\tableofcontents

\addtocontents{toc}{\protect\setcounter{tocdepth}{1}}

\section{Introduction}

A \emph{torus manifold} is a closed, orientable $2n$-dimensional manifold $M$ equipped with an effective $T^n$-action such that the fixed point set $M^{T}$ is non-empty. These spaces occupy a central place in algebraic topology and geometry, providing a rich class of highly structured manifolds whose global properties are deeply governed by the toric symmetry. A prominent subclass consists of the so called \emph{quasitoric manifolds}, where the torus action is furthermore locally standard and the orbit space is a simple polytope (see Definition \ref{definition quasitoric manifold}). A fundamental question regarding such manifolds is whether their symmetries can be realized via highly rigid Riemannian metrics. In this context, an isometric action of a Lie group $G$ on a Riemannian manifold $M$ is called \emph{polar} if there exists a section: an immersed submanifold that meets every orbit orthogonally. The existence of such a section reduces many deep geometric and analytic problems on $M$ to the lower-dimensional, often flat, geometry of the section itself. This makes polar actions a particularly rigid and geometrically rich class of isometric actions. The following result by Podestà and Thorbergsson connects them with Kähler geometry:

\begin{thm*}[{\cite[Theorem 1.2]{PodestaThorbergsson2002Coisotropic}}] An effective isometric action of a torus $T^n$ on a compact Kähler manifold $M$ with fixed points and complex dimension $n$ is always polar.
\end{thm*}

This is also intertwined with other results in symplectic and toric topology. First, since the action is isometric on a compact Kähler manifold, it is by holomorphic isometries \cite[p. 247]{KobayasgiNomizu1963}. The original statement of \cite[Theorem 1.2]{PodestaThorbergsson2002Coisotropic} supposes $\chi(M) > 0$, but this hypothesis is only used to guarantee fixed points. By results of Frankel, such a torus action is Hamiltonian \cite[Lemmas 1 and 2]{Frankel1059}, making $M$ a symplectic torus manifold. By Delzant's classification, the image of the moment map is a simple convex polytope, naturally identified with the orbit space $M/T^n$ \cite{Delzant1988}. Furthermore, the action is locally standard \cite[Proof of Lemme 2.4]{Delzant1988}. In other words, under the Podestà and Thorbergsson hypotheses, $M$ is quasitoric.

A natural motivating question for our work is whether this connection extends to more general topological and Riemannian settings. Indeed, we show that the Kähler constraints are unnecessary:

\begin{thmx}[Corollary {\ref{cor-toricpolar}}]\label{theorem A}
Every quasitoric manifold admits an invariant polar metric.
\end{thmx}

This provides a canonical and rigid geometric realization for a broad class of spaces traditionally studied through the lens of algebraic topology. However, Theorem \ref{theorem A} is in fact a specific application of a much broader geometric principle. The infinitesimal counterpart of polarity is fundamentally local and representation-theoretic: an action is \emph{infinitesimally polar} if all of its isotropy representations on the normal spaces to the orbits are polar. Under what conditions does an infinitesimally polar action admit an invariant Riemannian metric that makes it polar? We resolve this question completely for the case of torus actions. By investigating the natural principal $T$-orbibundle associated with the Grassmannian blow-up of infinitesimal sections, we identify the precise cohomological obstruction to global polarity:

\begin{thmx}[Theorem \ref{thm-infpolar}] An infinitesimally polar $T$-manifold is polar for some $T$-invariant metric if and only if its real Euler class $e_\RR(M,T)$ vanishes.
\end{thmx}

Since simple polytopes are contractible, their real orbifold cohomology trivially vanishes, yielding Theorem A.

\section{Polar actions and quasitoric manifolds} \label{sec-intro}

All actions are assumed to be differentiable. All manifolds are assumed to be connected. Let $G$ be a Lie group acting on $M$. We denote by $\pi \colon M\to M/G$ the canonical projection.

\begin{defn}
  Let $M$ be a Riemannian manifold on which $G$ acts by isometries. Let $q$ be the cohomogeneity. The action is called {\em polar} if there is an immersed submanifold $\Sigma$ of dimension $q$ that meets all orbits such that the intersection is everywhere orthogonal.
\end{defn}
\begin{ex}
    Examples of polar actions are conjugation of compact Lie groups, isotropy actions and Hermann actions on symmetric spaces.
\end{ex}
\begin{defn}
  Let $M$ be a $G$-manifold. The action is called {\em infinitesimally polar} if all isotropy representations $\actionarrow{G_x}{T_xM/T_xG}$ are polar with respect to some inner product.
\end{defn}
Originally an isometric action on a Riemannian manifold $(M,g)$ is said to be infinitesimally polar if all isotropy representations are polar with respect to the {\em induced} inner products \cite{Lytchak2010Curvature}. Our seemingly weaker, more topological definition is in fact equivalent for proper actions.

\begin{prop}\label{prop:metric}
    Let $(V,G)$ be a polar representation with respect to an inner product $\langle \ , \ \rangle$. Then any other $G$-invariant inner product is also polar with the same sections.
\end{prop}

\begin{proof}
    Let $\llangle \ , \ \rrangle$ be another $G$-invariant inner product on $V$. The linear map $A \colon V\to V$ defined by $\langle X,Y \rangle=\llangle AX,Y\rrangle$ is self-adjoint and positive definite with respect to $\llangle \ , \ \rrangle$. Therefore the decomposition $V=\bigoplus E_i$ of $V$ into eigenspaces $E_i$ of $A$ with eigenvalues $\lambda_i$ is orthogonal with respect to $\llangle \ , \ \rrangle$, and consequently for $\langle \ , \ \rangle$ as well. The $G$-invariance of both inner products implies that $A$ is $G$-equivariant, i.e., $gA=Ag$ for every $g\in G$. Thus the $E_i$ are invariant under the action. We fix $i$. The action on $E_i$ is polar by \cite[Theorem 4(i)]{Dadok} with respect to the first inner product. Since $A|_{E_i}=\lambda_i\id_{E_i}$, the second inner product is a scalar multiple of the first on $E_i$. Therefore, the representation $(E_i,G)$ is also polar with respect to the second inner product, and the sections $\Sigma_i$ on $E_i$ are identical. 

   We will now see that a section $\Sigma$ of $(V,G)$ is of the form $\Sigma=\oplus \Sigma_i$ for both scalar products. By \cite[Theorem 4(ii)]{Dadok}, there are Lie groups $H_i$ and polar representations $(E_i,H_i)$  such that the representation $(V,G_0)$ is orbit-equivalent to the product representation $\bigoplus(H_i,E_i)$. In particular, $(E_i,H_i)$ are orbit equivalent to $(E_i,G_0)$ and therefore have the same sections $\Sigma_i$. Hence the sections of $(V,G_0)$ and thus of $(V,G)$ are of the form $\Sigma=\oplus \Sigma_i$. 
\end{proof}

To conclude the equivalence of both definitions of infinitesimally polar actions for proper actions, we observe the following. Consider a proper $G$-action that is infinitesimally polar in the sense of the above definition. Because the action is proper, the manifold admits a $G$-invariant Riemannian metric $g$, and the isotropy group $G_x$ at any point is compact. By hypothesis, the isotropy representation of $G_x$ on the normal space $\nu_x(Gx)$ is polar with respect to some inner product. Since $g_x$ is a $G_x$-invariant inner product, Proposition \ref{prop:metric} guarantees that the representation is also polar with respect to $g_x$. Thus, the action is infinitesimally polar in the original Riemannian sense.
\color{black}

\begin{ex}[Standard model]
  The $T^n$-action on $\mathbb{C}^n$ by coordinatewise complex multiplication, called the standard representation of $T^n$, is polar with section $\RR^n$ and Weyl group $\ZZ^n$ acting by reflections. The orbit space is the the good orbifold 
  \[\ \RR^n/\ZZ^n=\mathbb{R}_{\geq}^n=\{x\in\mathbb{R}^n \mid x_i\geq 0\ \text{for all}\ i\}.\]
\end{ex}

\begin{defn}[{\cite{DavisJanuszkiewicz1991}}]\label{definition quasitoric manifold} An action of a torus $T^{n}$ on a $2n$-dimensional manifold $M^{2n}$ is \emph{locally standard} if every point has a $T^n$-invariant neighborhood that is equivariantly diffeomorphic, up to an automorphism of $T^n$, to an open $T^n$-invariant subset of $\mathbb{C}^{n}$ with the standard linear $T^{n}$-action. More precisely, this means there exists a diffeomorphism $f$ and an automorphism $\rho \in \text{Aut}(T^n)$ such that $f(t x) = \rho(t)  f(x)$). A \emph{quasitoric manifold} is a closed $2n$-dimensional manifold equipped with a locally standard $T^n$-action such that the orbit space $M/T$ is an $n$-dimensional simple polytope.
\end{defn}

Davis and Januszkiewicz introduced quasitoric manifolds as a purely topological analogue to smooth toric varieties. They showed that one can construct them directly from a combinatorial object: a simple convex polytope $P$ equipped with a \emph{characteristic function} which assigns a subgroup of $T^n$ to each facet of $P$. The quasitoric manifold is then constructed by coherently gluing the data, in analogy with Delzant’s construction of toric varieties from fans.

\begin{prop}\label{prop:quasitoric infinitesimally polar}
    Let $M^{2n}$ be a quasitoric manifold. Then $\actionarrow{T^n}{ M^{2n}}$ is infinitesimally polar. Moreover, $M/T$ is a Coxeter orbifold.
\end{prop}

\begin{proof}
    If $\actionarrow{T^n}{M^{2n}}$ is locally standard, the isotropy representation at $x$ is equivalent to the isotropy representation at some corresponding point $z \in \mathbb{C}^n$. Up to a permutation of coordinates, let $z = (z_1, \dots, z_n) \in \mathbb{C}^n$ be a point where the first $k$ coordinates are zero and the remaining $n-k$ coordinates are non-zero. The isotropy group at $z$ is the standard coordinate torus $T^k \subset T^n$. The normal space $\nu_z T^n  z$ splits as $\mathbb{C}^k \oplus \mathbb{R}^{n-k}$, where the $\mathbb{R}^{n-k}$ factor corresponds to the radial directions of the non-zero complex coordinates. 
    
    The isotropy representation of $T^k$ on this normal space acts in the standard way on $\mathbb{C}^k$ and trivially on $\mathbb{R}^{n-k}$. The standard $\actionarrow{T^k}{\mathbb{C}^k}$ is a classic polar representation with section $\mathbb{R}^k$. Thus, the full isotropy representation on $\nu_z(T^n z)$ is polar with section $\mathbb{R}^k \oplus \mathbb{R}^{n-k} \cong \mathbb{R}^n$. Because every isotropy representation of $M^{2n}$ is equivalent to one of these polar representations, the action on $M^{2n}$ is infinitesimally polar.

    Now let us show that $M/T$ is a Coxeter orbifold. Because the $T^n$-action on $M^{2n}$ is locally standard, the local structure of $M/T$ is entirely determined by the standard $\actionarrow{T^n}{\mathbb{C}^n}$. The orbit map for the standard action, given by $(z_1, \dots, z_n) \mapsto (|z_1|^2, \dots, |z_n|^2)$, identifies $\mathbb{C}^n / T^n$ with
    \[\mathbb{R}_{\geq}^n=\{x\in\mathbb{R}^n \mid x_i\geq 0\ \text{for all}\ i\},\]
    which can be naturally viewed as the quotient of $\mathbb{R}^n$ by the canonical action of $\mathbb{Z}_2^n$ where the generators act by isometric reflections across the coordinate hyperplanes. Consequently, any open subset of $\mathbb{R}_{\geq}^n$ is modeled on the quotient of an open set in $\mathbb{R}^n$ by a finite group generated by reflections. Since every point in $M/T$ has a neighborhood homeomorphic to an open set in $\mathbb{R}_{\geq}^n$, the local models of $M/T$ are quotients of Euclidean space by finite reflection groups.
\end{proof}

\section{Grassmannian blow-ups}



We start by noticing the following application of considerations by Lytchak in \cite{Lytchak2010Geometric}:

\begin{prop}
    If a proper action $\actionarrow{G}{M}$ is infinitesimally polar, then $M/G$ is an orbifold.
\end{prop}

Note that properness ensures there is a $G$-invariant metric $g$ on $M$. Then the infinitesimally polar action in our topological sense is infinitesimally polar in the Riemannan sense of Lytchak (with respect to $g$), by Proposition \ref{prop:metric}.

We will now consider an effective, infinitesimally polar $T^n$-manifold $M^m$. At the moment not necessarily $m=2n$; we denote by $q$ the codimension $m-n$. Hence $M/T$ is an orbifold, as we just saw. We fix a $T$-invariant metric $g$ on $M$. The isotropy representations $\actionarrow{T_x}{\nu_xGx}$ on the normal bundles with respect to the chosen metric are polar. Its (linear) sections are called infinitesimal sections. Let $\whM$ be the set of all infinitesimal sections with respect to the chosen metric. If $x$ is a regular point, then the infinitesimal section through $x$ coincides with $\nu_x Tx$. We call $\whM$ the Grassmannian blow-up of $(M,T,g)$ and let $\rho \colon \whM \to M$ be the projection. It is a closed submanifold of the Grassmannian $G_q(M)$ of $q$-planes of $TM$ (see \cite[Section 4.2]{Lytchak2010Geometric} for details). Up to diffeomorphism it is independent of the choice of an invariant metric. Then the natural action of $T$ on $\whM$ is differentiable and locally-free. The map $\rho$ is $T$-equivariant, hence $\whM/T=M/T$. The projection $\hat\pi \colon \whM\to \whM/T=M/T$ is a $T$-orbibundle.

We will consider the {\em Euler class}
\[e=e(M,T)\in H^2_{\text{orb}}(M/T,\ZZ^n),\]
which classifies the principal $T$-orbibundle $\hat\pi \colon \whM\to M/T$, and the {\em real Euler class}
\[e_\RR\in H^2_{\text{orb}}(M/T,\RR^n),\]
which is the image of $e$ via the natural coefficient homomorphism $H^2_{\text{orb}}(M/T,\ZZ^n) \to H^2_{\text{orb}}(M/T,\RR^n)$. It is worth noticing that, by a standard result of Satake (\cite[Theorem 1]{satake}), one has the isomorphism $H^2_{\text{orb}}(M/T,\RR^n)\cong H^2(M/T,\RR^n)$, between orbifold and singular cohomology with real coefficients. Thus, we may seamlessly view $e_\RR$ as an element of $H^2_{\text{orb}}(M/T,\RR^n)\cong H^2_{\text{orb}}(M/T)\otimes \RR^n$.

Since the action is locally free, there is a $T$-invariant $\mathfrak{t}$-valued connection form $\omega\in\Omega^1( \hat{M})\otimes \mft$ (see \cite[Appendix A2, Section 3]{guilleminGinzburgKarshon}). Since $\hat\pi^*$ provides an isomorphism (of differential graded algebras) between $\Omega^*(M/T)\otimes \mft$ and the complex $\Omega_{\text{bas}}(\whM, T)\otimes \mft$ of $\mft$-valued $T$-basic forms on $\whM$, and $d\omega$ is $T$-basic, there is a unique $\Omega\in \Omega^2(M/T)\otimes \mft$ with $\pi^*\Omega=d\omega$, called the curvature $2$-form of the orbibundle $\hat\pi$. By classical Chern--Weil theory adapted to orbibundles, 
\[e_\RR(M,T) = [\Omega].\] Remember that the $T$-invariant connection form $\omega$ corresponds one-to-one to a $T$-invariant transverse distribution, the associated {\em horizontal distribution} which is integrable if and only if $\Omega=0$. Moreover, if $[\Omega]=0$, then there is a connection $\omega$ such that the horizontal distribution $\ker\omega$ is integrable. We are now in position to state our main results.

\begin{thm}\label{thm-infpolar}
  An infinitesimal polar $T$-manifold is polar for some $T$-invariant metric if and only if its real Euler class vanishes.
\end{thm}
We prove this theorem in Section \ref{section main proof}.


\begin{cor}\label{cor-toricpolar}
  A quasitoric manifold admits an invariant Riemannian metric for which the action is polar. Moreover, the orbit space is a Riemannian Coxeter orbifold with respect to the induced metric.
\end{cor}


\begin{proof}
  A quasitoric manifold is locally standard, therefore infinitesimally polar, by Proposition \ref{prop:quasitoric infinitesimally polar}. Since $P=M/T$ is contractible, $H^2_{\text{orb}}(M/T,\RR^n)\cong H^2(M/T,\RR^n)=0$ and consequently the real Euler class is trivial.
\end{proof}

\section{Proof of Theorem \ref{thm-infpolar}}\label{section main proof}
Before we come to the proof, we look at the tubular neighborhoods for the action of $T$ both on $M$ and $\whM$ and their local orbit spaces. Assume that $\actionarrow{T}{M}$ is infinitesimally polar and let $g$ be a $T$-invariant metric on $M$.
Let $\varepsilon_i>0$ be so small that the normal exponential map $\exp^\perp:\nu^{\varepsilon_i} Tp_i\to M$ of $Tp_i$ is a diffeomorphism onto its image $U_i$. Here $\nu^{\varepsilon_i}$ denotes the normal vectors of length smaller than $\varepsilon_i$. The tubular neighborhood $U_i$ of $Tp_i$ is $T$-equivariantly identified with the disc bundle $\nu^{\varepsilon_i}Tp_i\cong T \times_{T_{p_i}} \nu^{\varepsilon_i}_{p_i} T p_i$. We consider the Grassmannian blow-up $\whM$ with respect to $g$. Let $\Sigma_i$ be a (flat) section in the slice $\nu^{\varepsilon_i} Tp_i$ with respect to the inner product $g_{p_i}$. Let $\hat p_i$ be the corresponding element in $\whM$. Let $\hat{U}_i:=\rho^{-1} (U_i)$. We consider the following commutative diagram, where the vertical lines pass to the respective orbit spaces, which are locally equivalent to the orbit spaces of the slice representations:
    $$\begin{tikzcd}[column sep=large]
    \hat{p}_i \in \hat{U}_i \subset \whM \arrow[r, "\rho"] \arrow[d, "\hat{\pi}"'] & M \supset U_i \ni p_i \arrow[d, "\pi"]\\
    \nu_{\hat{p}_i}^{\varepsilon_i}(T \hat{p}_i) / T_{\hat{p}_i} \subset \whM/T \arrow[r, "\tilde{\rho}"'] & M/T \supset U_i/T \cong \nu^{\varepsilon_i}T p_i/T_{p_i}\cong \Sigma_i / W(\Sigma_i).
\end{tikzcd}$$
The equivariant map $\rho$ induces the map $\tilde{\rho}$ on the quotients. Notice that $\actionarrow{T_{\hat p_i}}{\nu_{\hat p_i}Tp_i}$ and $\actionarrow{W(\Sigma_i)}{\Sigma_i}$ are equivalent representations. To see this, observe that the element $\hat{p}_i \in \whM$ corresponds to the section $\Sigma_i$. Since $T$ is abelian, the stabilizer $T_{\hat p_i}$ consists of the elements $t \in T$ that preserve $\Sigma_i$, meaning $T_{\hat{p}_i} = N_T(\Sigma_i)$. Because the action on $\whM$ is locally free, $T_{\hat p_i}$ is discrete and thus identifies with the effective action of the generalized Weyl group $W(\Sigma_i) = N_T(\Sigma_i)/Z_T(\Sigma_i)$. Thus $T_{\hat{p}_i}=W(\Sigma_i)$, because $Z_T(\Sigma_i)$ is trivial by effectiveness of the $T$-action. From this, we conclude that $\hat M/T\cong M/T$, as orbifolds.

We will now prove Theorem \ref{thm-infpolar}.

\begin{proof}
We claim that the following conditions are equivalent:
\begin{enumerate}
  \item The real Euler class $e_\RR(M,T)\in H^2_{\text{orb}}(M/T,\RR^n)$ vanishes,
  \item there is a flat $T$-connection of the $T$-orbibundle $\hat \pi \colon \whM \to \whM/T=M/T$,
  \item there is an {\em integrable} $T$-invariant distribution transverse to the $T$-orbits of $\whM$,
  \item there is a $T$-invariant set of transversals of $(\whM,T)$,
  \item there is a polar metric for $(M,T)$.
\end{enumerate}

Conditions (1) through (4) are equivalent by standard differential geometry: by Chern-Weil theory, the real Euler class vanishes if and only if the orbibundle admits a flat connection (giving $(1) \Leftrightarrow (2)$). Geometrically, a flat connection of a $T$-orbibundle corresponds exactly to an integrable, $T$-invariant horizontal distribution transverse to the fibers (which is $(2) \Leftrightarrow (3)$). By the Frobenius theorem, this distribution integrates to a $T$-invariant foliation by transverse submanifolds (hence $(3) \Leftrightarrow (4)$).

To see that $(5)\Rightarrow (4)$, assume there is a polar metric for $(M,T)$ and let $\whM$ be the corresponding Grassmannian blow-up. Because the metric is polar, it admits a family of global sections. The natural lifts of these sections $\Sigma$ into the Grassmannian, defined by $\hat{\Sigma} = \{ T_p\Sigma \mid p \in \Sigma \}$, form horizontal leaves that are everywhere transverse to the $T$-orbits on $\whM$. Since the $T$-action maps sections to sections, this family of horizontal leaves constitutes a $T$-invariant set of transversals for $(\whM, T)$.

Finally, let us prove $(4)\Rightarrow (5)$. For a leaf $\hat N$ of the horizontal distribution, the projection $N = \rho(\hat N)$ is an immersed submanifold. In fact, for any $\hat p \in \hat N$, one has the direct sum $T_{\hat p}\hat N \oplus T_{\hat p}(T \hat p)=T_{\hat p}\whM$ and $T_{\hat p}(T \hat p)=\mft\hat p \supset \mft_{p} \hat p = \ker d\rho_{\hat p}$, where $p:=\rho(\hat p)$, because $\rho^{-1}(p)=T_p\hat p$, the orbit of the isotropy group $T_p$ through $\hat p$. The mutual inclusion of the latter is a consequence on the one side of the equivariance of $\rho$ and on the other side of the transitivity of the $T_p$-action on the set of infinitesimal sections through $p$. Thus $T_{\hat p}\hat N\cap \ker d\rho_{\hat p}=\{0\}$ and $N$ is an immersed submanifold of $M$. Our aim is to construct a global $T$-invariant metric $g$ for which $N$ is orthogonal to the orbits; the action of $T$ on $(M,g)$ will then be polar with section $N$.

Choose a $T$-invariant Riemannian metric on $M$. Let $\{U_i\}$ be a locally finite $T$-invariant covering of $M$ by tubular neighborhoods of orbits $T p_i$ as previously. Recall the $T$-equivariant identification $U_i\cong T \times_{T_{p_i}} \nu^{\varepsilon_i}_{p_i} T p_i$. Since the action is infinitesimally polar, the representation $\actionarrow{T_{p_i}}{\nu_{p_i}^{\varepsilon_i}T p_i}$ is polar with respect to the induced inner product. Let $\Sigma_i \subset \nu_{p_i}^{\varepsilon_i}T p_i$ be a section for this linear action. We equip $\nu_{p_i}^{\varepsilon_i}T p_i$ with the corresponding flat metric, and $T$ with an invariant metric. Because $T_{p_i}$ acts isometrically on the product $T \times \nu_{p_i}^{\varepsilon_i}T p_i$, there is a unique $T$-invariant Riemannian metric $h_i$ on the quotient $U_i \cong T \times_{T_{p_i}} \nu^{\varepsilon_i}_{p_i} T p_i$ such that the natural projection is a Riemannian submersion. Crucially, because $\Sigma_i$ intersects the linear $T_{p_i}$-orbits orthogonally, the tangent bundle of the product $\{1\} \times \Sigma_i$ consists entirely of horizontal vectors with respect to this submersion. Therefore, the projection of $\{1\} \times \Sigma_i$ yields a section that is everywhere orthogonal to the $T$-orbits in $U_i$, meaning $U_i$ is $h_i$-polar. In fact, this section coincides with the section $\Sigma_i$ of $\nu^{\varepsilon_i}_{p_i}T p_i$ considered as contained in $U_i$.

Let $\hat N$ be a horizontal leaf in $\whM$. For each $U_i$, let $\hat p_i\in \hat N\cap \rho^{-1}(p_i)$ and let $\Sigma_i$ be the corresponding infinitesimal section. Let $\hat N_i$ be the connected component of $\hat N$ in $\hat U_i$ containing $\hat p_i$. We define the local open set $\hat U_i := \rho^{-1}(U_i)$ in $\whM$. Let $\hat \Sigma_i'=\{T_p\Sigma_i\mid p\in\Sigma_i\}$ be the canonical lift of the section $\Sigma_i$ to the blow-up $\hat U_i'$ of $U_i$ with respect to the metric $h_i$ and let $\hat\Sigma_i:= I_{h_ig}(\hat \Sigma_i')$, where $I_{h_i,g}$ is as defined in \cite[Section 4.2]{Lytchak2010Geometric}, sending $\hat U_i'$ to $\hat U_i\subset \hat M$. Since both $\hat N_i$ and $\hat \Sigma_i$ are local transversals for the $T$-orbits in $\hat U_i$, there exists a unique, smooth $T$-equivariant diffeomorphism $\hat\varphi_i \colon \hat U_i \to \hat U_i$ that maps $\hat N_i$ to $\hat \Sigma_i$ while preserving the orbits; in other words it's a gauge transformation for $\hat\pi|\hat U_i$. Now, since the action on $\hat U_i$ is locally free, we can apply Theorem 3.1 of \cite{HS 1991}: it guarantees the existence of a smooth map $F_i \colon  \hat U_i/T \to T$ such that $\hat\varphi_i(\hat x) = F_i(\hat\pi(\hat x)) \hat x$. Recall that the map $F_i$ on the quotient space is called smooth if $F_i\circ \hat \pi$ is smooth.

Identifying $\hat U_i/T = U_i/T$, we define a map $\varphi_i \colon  U_i \to U_i$ on the base manifold by $\varphi_i(x) = F_i(\pi(x)) x$. Because $F_i$ and the orbit projection $\pi$ are smooth, $\varphi_i$ is a well-defined, smooth $T$-equivariant diffeomorphism. Furthermore, because $\hat\pi = \pi \circ \rho$, we have $\rho \circ \hat\varphi_i = \varphi_i \circ \rho$. Since $\hat \varphi_i$ maps $\hat N_i$ to $\hat \Sigma_i$, it follows that $\varphi_i$ maps $N_i$, the projection of the plaque $\hat N_i$, smoothly to the section $\Sigma_i$. This follows from $\rho \circ \hat\varphi_i = \varphi_i \circ \rho$ and $\rho(\hat \Sigma_i)=\rho(I_{h_ig}(\hat \Sigma_i'))=\rho(\hat \Sigma_i')=\Sigma_i$.

We now define a new local metric on $U_i$ by the pullback $g_i := \varphi_i^* h_i$. Because $\varphi_i(N_i)=\Sigma_i$, and $\Sigma_i$ is $h_i$-orthogonal to the $T$-orbits in $U_i$ by construction, it follows that $N_i$ is $g_i$-orthogonal to the $T$-orbits. Finally, let $\{\phi_i\}$ be a $T$-invariant partition of unity subordinate to $\{U_i\}$. We define the global $T$-invariant metric $g := \sum_i \phi_i g_i$. Because $N$ is orthogonal to the $T$-orbits with respect to every individual metric $g_i$, it remains orthogonal to the orbits with respect to $g$. The transversal $N$ meets every orbit, and this intersection is everywhere orthogonal, proving that $\actionarrow{T}{M}$ is polar with respect to $g$.\end{proof}

\end{document}